\documentclass[11pt,a4paper]{amsart}
\usepackage[dvips]{graphicx}
\usepackage{amsmath}
\usepackage{amscd}
\usepackage{amssymb}
\usepackage{latexsym}
\usepackage[T1]{fontenc}
\usepackage[latin1]{inputenc}
\usepackage[all]{xy}

\newtheorem{Prop}{Proposition}[section]
\newtheorem{Def}[Prop]{Definition}
\newtheorem{Theorem}[Prop]{Theorem}
\newtheorem{Lemma}[Prop]{Lemma}

\newtheoremstyle{Example}{\topsep}{\topsep}%
  {}%         Body font
  {}%         Indent amount (empty = no indent, \parindent = para indent)
  {\bfseries}% Thm head font
  {.}%        Punctuation after thm head
  { }%     Space after thm head (\newline = linebreak)
  {\thmname{#1}\thmnumber{ #2}\thmnote{ #3}}%         Thm head spec

\theoremstyle{Example}
\newtheorem{Example}[Prop]{Example}
\newtheoremstyle{Remark}{\topsep}{\topsep}%
  {}%         Body font
  {}%         Indent amount (empty = no indent, \parindent = para indent)
  {\bfseries}% Thm head font
  {.}%        Punctuation after thm head
  { }%     Space after thm head (\newline = linebreak)
  {\thmname{#1}\thmnumber{ #2}\thmnote{ #3}}%         Thm head spec

\theoremstyle{Remark}
\newtheorem{Remark}[Prop]{Remark}

\title[Deformations of polarizations of powers of the maximal ideal]{Deformations of polarizations of powers of the maximal ideal}
\author{Henning Lohne}
\address{Matematisk Institutt\\
         Johs. Brunsgt. 12\\
         5008 Bergen}
\email{henning.lohne@math.uib.no}
\date{\today}

\begin{document}

\begin{abstract}
In this paper, we study the two natural polarizations, namely the standard polarization and the box polarization, of the $d$-th power of the maximal ideal in a polynomial ring. We show that these polarizations correspond to smooth points in the Hilbert scheme, and we calculate the dimension of their component which shows that they lie on different components. When $d=2$, we show that all maximal polarizations are smooth points, and we give a simple method for calculating the dimension of their component. 
\end{abstract}

\bibliographystyle{plain}

\maketitle

\section{Introduction}

Let $k$ be a field, and let $S=k[x_1,\dots,x_n]$ be the polynomial ring in $n$ variables. We are interested in different polarizations of the ideal $m^d$. In \cite{NR2009Betti}, Uwe Nagel and Victor Reiner defined the complex of boxes resolution, which gives a natural minimal resolution of a certain polarization of the ideal $m^d$. Yaganawa showed in \cite{yanagawa2012alternative} that this natural polarization, which we will call the box polarization of $m^d$, can be defined for any Borel ideal. In this paper, we study the Hilbert scheme of this, and other polarizations of $m^d$. We show that both the standard polarization and the box polarization give smooth points on the Hilbert scheme. In the special case $d=2$, we have that all maximal polarizations of $m^2$ are actually deformations of another related ideal $m_\mathrm{sq.fr.}^2$. Maximal polarizations of such ideals are classified by spanning trees of the complete graph. The deformations of such a polarization can be calculated. We show that all such polarizations are smooth points on the Hilbert scheme, and we calculate the dimension of the component it lies on. 

\medskip
The paper is organized as follows: Section \ref{polarizationsanddeformations} contains the definitions of what we mean by a polarization, and some basic fact about deformations. We also show how it is possible to calculate the dimension of the tangent space of the box polarization and the standard polarization in the Hilbert scheme of a power of the maximal ideal. 

In Section \ref{theboxpolarization}, we calculate the dimension of the tangent space of the box polarization in the Hilbert scheme. We also show that this ideal lies on a component where the dimension is known. Since these two dimensions are equal, it follows that the box polarization is a smooth point on the Hilbert scheme. 

In Section \ref{standardpolarization}, we find all first order deformations of the standard polarization. If $n\ge 4$, then these are only deformations of the variables. In this case, the standard polarization is a smooth point in the Hilbert scheme. When $n=2$, then the standard polarization and the box polarization are the same, so it is also a smooth point. Finally, for $n=3$, we show that there in fact are three first order deformations which are not deformations of the variables. In this case, we explicitely shows that they can be liftet to global deformations, which shows that the standard polarization is also a smooth point in the Hilber scheme when $n=3$.

In Section \ref{kvadratfriversjon}, we study the square-free ideal $m_\mathrm{sq.fr}^d$, and maximal polarizations of this ideal. Such polarizations corresponds to spanning trees of the complete graph. By using this correspondence we manage to calculate all the deformations of such polarizations, and we show that they are all smooth points in the Hilbert scheme. We also give an easy algorithm for calculating the dimension of the component they lie on, by using the corresponding spanning tree.

\medskip
\noindent
\textbf{Acknowledgments.} I would like to thank Professor Gunnar Fl\o ystad for giving me valuable comments and suggestions regarding the work of this paper. 

\section{Polarizations and deformations}\label{polarizationsanddeformations}
\begin{Def}\label{polarization}
Let $I$ be an ideal in $S=k[x_1,\dots,x_n]$. A polarization $\widetilde{I}$ of $I$ is an ideal in the polynomial ring
$$\widetilde{S}:=k\left[x_{11},\dots,x_{1r_1},x_{21},\dots,x_{2r_2},\dots,x_{nr_n}\right]$$
such that the sequence
$$\sigma = \left(x_{11}-x_{12}, x_{11}-x_{13},\dots, x_{11}-x_{1r_1},x_{21}-x_{22},\dots, x_{n1}-x_{nr_n}\right)$$
is a regular $\widetilde{S}/\widetilde{I}$-sequence, and that $\widetilde{I}\otimes \widetilde{S}/ \langle \sigma\rangle \cong I$. The homomorphism $\varphi: \widetilde{I} \longrightarrow I$ is called the depolarization of $\widetilde{I}$. 
\end{Def}

\begin{Def}
Let $m^d$ be the $d$-th power of the maximal ideal in the polynomial ring $S=k[x_1,\dots,x_n]$, and let $P_d$ be a polarization of $m^d$, and let $\varphi: P_d\rightarrow m^d$ be the depolarization homomorphism. The preimages $\varphi^{-1}(x_i^d)$ are called the vertices of $P_d$.
\end{Def}

We want to study and compare different polarizations of $m^d$. We therefore want them to be ideals in the same polynomial ring. Every polarization of $m^d$ can be identified by an ideal $P_d$ in the polynomial ring
$$\widetilde{S}=k[x_{11},\cdots, x_{1d}, \cdots, x_{n1}\cdots x_{nd}],$$
and the Betti numbers and Hilbert polynomial of the rings $\widetilde{S}/P_d$ are all the same. We define the trivial deformation to be $M_d:=(x_{11},x_{21},\dots,x_{n1})^d$. In this paper we want to study the Hilbert scheme that classifies all closed subschemes $Y\subseteq \mathrm{Proj}(\widetilde{S})$ having Hilbert polynomial equal to the Hilbert polynomial of $\widetilde{S}/M_d$. So every polarization of $m^d$ corresponds to a point in this Hilbert scheme. We will show that two natural polarizations of $m^d$, namely the standard polarization and the box polarization of $m^d$ are both smooth points on this Hilbert scheme.

We recall some notion from deformation theory. Let $D=k[t]/t^2$ denote the dual numbers, and let $S'=S[t]/t^2$. A first order deformation of $I\subset S$ is an ideal $I'\subset S'$ that extends $I$ over the dual numbers. In other words, it is an ideal $I'\subset S'$ such that $S'/I'$ is flat over $D$ and such that $k\otimes_D (S'/I') \cong S/I$.

The set of first order deformations can be given a module structure by defining $T^0:=\mathrm{Hom}_S(I,S/I)$. This correspondence and more details can be found in Section 1.2, or more precisely Proposition 2.3 in \cite{hartshorne2009deformation}.

An ideal $I'\subset S'$ is a first order deformation of $I$ if and only if every relation of $I$ lifts to a relation of $I'$. More precisely, we have a well know fact presented in the following lemma:

\begin{Lemma}\label{liftrelations} Let $I$ be an ideal in $S$, and let $I'$ be an ideal of $S'$ such that $I'\otimes_{S'} S \cong I$. Assume that we have surjections $\varphi:S^p\rightarrow I$ and $\varphi':{S'}^p\rightarrow I'$. Then $I'$ is a first order deformation of $I$ if and only if every relation of $I$, i.e. an element of $\mathrm{ker}(\varphi)$, lifts to a relation of $I'$, i.e. an element of $\mathrm{ker}(\varphi')$.
\end{Lemma}

\begin{proof}
By \cite[Proposition 2.2]{hartshorne2009deformation}, the fact that $I'$ is a first order deformation is equivalent to the sequence
$$0\longrightarrow S/I\overset{t}{\longrightarrow} S'/I' \longrightarrow S/I \rightarrow 0$$
being exact. By the $9$-lemma, this is equivalent to sequence
$$0\longrightarrow I\overset{t}{\longrightarrow} I' \longrightarrow I \rightarrow 0$$
being exact. By using the $9$-lemma again, we get
$$
\xymatrix{
 & 0 \ar[d] & 0 \ar[d] & 0 \ar[d] & \\
0 \ar[r] & \mathrm{ker}(\varphi) \ar[r]^t \ar[d] & \mathrm{ker}(\varphi') \ar[r] \ar[d] & \mathrm{ker}(\varphi) \ar[r] \ar[d] & 0 \\
0 \ar[r] & S^p \ar[r]^t \ar[d] & {S'}^p \ar[r] \ar[d] & S^p \ar[r] \ar[d] & 0 \\
0 \ar[r] & I \ar[r]^t \ar[d] & I' \ar[r] \ar[d] & I \ar[r] \ar[d] & 0 \\
 & 0 & 0 & 0 &
}
$$
and the sequence of the bottom is exact if and only if the sequence on the top is exact. Since the map $\mathrm{ker}(\varphi)\overset{t}{\rightarrow} \mathrm{ker}(\varphi')$ is clearly injective, we get that $I'$ is a first order deformation if and only if the map $\mathrm{ker}(\varphi')\rightarrow \mathrm{ker}(\varphi)$ is surjective. That is, if and only if every relation of $I$ lifts to a relation of $I'$. 
\end{proof}

We now want to calculate the dimension of the tangent space of the different polarizations of $m^d$.
\begin{Prop}
Let $S=k[x_{11},\dots,x_{nd}]$ and let $I\subset S$ be a polarization of the ideal $m^d$. If $n\ge 2$ and $d\ge 2$, then the dimension of the tangent space of $I$ in the Hilbert scheme equals $\mathrm{dim}_k (\mathrm{Hom}_S(I,S/I))_0$
\end{Prop}

\begin{proof}
If $X = \mathrm{Proj}(S)$ and $Y = \mathrm{Proj}(S/I)$, and if $H$ is the Hilbert scheme of $Y$ in $X$. Then the Zariski tangent space of $y\in H$, corresponding to $Y$ is $H^0(Y,\mathcal{N}_{Y/X})$. The normal sheaf $\mathcal{N}_{Y/X}$ is isomorphic to the sheaf $\mathcal{H}om_X(\mathcal{I}, \mathcal{O}_Y)$. By the long exact sequence
$$0\rightarrow H_m^0(M) \rightarrow M \rightarrow \sum_\nu H^0(\mathrm{Proj}(S/I), \widetilde{M}(\nu))\rightarrow H_m^1(M) \rightarrow 0,$$
which for example is found in \cite[Theorem A4.1]{eisenbud1995commutative}, we see that the global sections of $\mathcal{H}om_X(\mathcal{I}, \mathcal{O}_Y)$ can be identified with the vector space $(\mathrm{Hom}_S(I,S/I))_0$ if the depth of the module $\mathrm{Hom}_S(I,S/I)$ is greater than or equal to two. So that is what we want to show.

First of all, we observe that $\mathrm{depth}\,S/I\ge 2$. But this follows from Definition \ref{polarization}, since $(x_{11}-x_{12},\dots,x_{n1}-x_{nd})$ is a regular $S/I$-sequence, and if $n\ge 2$ and $d\ge 2$, it is clear that this sequence has length at least $2$.

By the Auslander--Buchbaum theorem (e.g \cite[Theorem 1.3.3]{bruns1998cohen}), we have that $\mathrm{depth}\,\mathrm{Hom}_S(I,S/I) = \mathrm{depth}\,S - \mathrm{pdim}\,\mathrm{Hom}_S(I,S/I)$, where pdim denotes the projective dimension. So $\mathrm{depth}\,\mathrm{Hom}_S(I,S/I)\ge 2$ if and only if $\mathrm{pdim}\,\mathrm{Hom}_S(I,S/I)\le nd-2$, where $nd=\mathrm{dim}\, S = \mathrm{depth}\, S$. By \cite[Corollary 1.3.2]{bruns1998cohen}, we have that $\mathrm{pdim}\,\mathrm{Hom}_S(I,S/I) = \mathrm{max}\{i\,|\,\mathrm{Tor}_i^S(\mathrm{Hom}_S(I,S/I), k)\neq 0\}$. So it is enough to verify that $\mathrm{Tor}_i^S(\mathrm{Hom}_S(I,S/I), k) = 0$ for $i=nd-1$ and $i=nd$. Consider a free presentation
$$S^r\longrightarrow S^p \longrightarrow I \longrightarrow 0.$$
This gives rise to a left exact sequence
$$0 \longrightarrow \mathrm{Hom}_S(I,S/I) \longrightarrow (S/I)^p \longrightarrow (S/I)^r,$$
and by extending with the cokernel of the last map, we get a long exact sequence
$$0 \longrightarrow \mathrm{Hom}_S(I,S/I) \longrightarrow (S/I)^p \longrightarrow (S/I)^r\longrightarrow C \longrightarrow 0.$$
This gives us two short exact sequences:
$$0 \longrightarrow \mathrm{Hom}_S(I,S/I) \longrightarrow (S/I)^p\longrightarrow K \longrightarrow 0,$$
and
$$0 \longrightarrow K \longrightarrow (S/I)^r\longrightarrow C \longrightarrow 0.$$
We can now tensor both these short exact sequences by $k$, and we get that $\mathrm{Tor}_{i}^S(\mathrm{Hom}_S(I,S/I), k)\cong \mathrm{Tor}_{i+1}^S(K,k)$ for $i\ge nd-1$, since $\mathrm{depth}\,(S/I)^p=\mathrm{depth}\,(S/I)\ge 2$. So $\mathrm{Tor}_{nd}^S(\mathrm{Hom}_S(I,S/I), k)=0$ because of Hilbert's syzygy theorem. Similarly, from the other sequence we also get that $\mathrm{Tor}_{i}^S(K,k)\cong \mathrm{Tor}_{i+1}^S(C,k)$, and Hilbert's syzygy theorem again gives us that $\mathrm{Tor}_{nd}^S(K,k)=0$. Using this, we also get that $\mathrm{Tor}_{nd-1}^S(\mathrm{Hom}_S(I,S/I), k)=0$, which completes the proof.

\end{proof}

If we fix a minimal generator set $\{f_i\}$ of $I$, then for each first order deformation $I'$ of $I$, we can find a minimal generator set $\{g_i=f_i+th_i\}$. We can now define the vector $\mathbf{v}_{I'} := [h_1,h_2,\dots,h_g]^T\in S^g$.

\begin{Prop}
The vector space $(\mathrm{Hom}_{S}(I, S/I))_0$ is isomorphic to the vector space spanned by the vectors $\mathbf{v}_{I'}$ of degree $d$ where $I'$ is a first order deformation of $I$.
\end{Prop}

\begin{proof}
Consider the surjection $\varphi:S^p\rightarrow I$ from the free presentation of $I$. If we apply the functor $\mathrm{Hom}_S(-,S/I)$, we get an injection $\mathrm{Hom}_S(I,S/I) \rightarrow \mathrm{Hom}_S(S^g,S/I) \cong (S/I)^g$. If $I'$ is a first order deformation of $I$, it corresponds to a homomorphism in $\mathrm{Hom}_S(I,S/I)$, so via the injection into $(S/I)^g$, it also corresponds to an element there. Its image is the coset of $\mathbf{v}_{I'}$, and if $I'$ corresponds to a homomorphism of degree $0$, this vector is of degree $d$.
\end{proof} 

\subsection{Calculations}
In the rest of this section, we will explicitly calculate the dimension of the tangent space of different polarizations $P_d$ in the Hilbert scheme of $M_d$ in the ring $\widetilde{S}$. We will do so by constructing all ideals that are generated by polynomials $g=f+th$, where $f$ is a generator of $P_d$ and $h$ is a monomial of degree $d$, such that the relations of $P_d$ lift to relations of this new ideal. In this case, we will say that the generator $g$ is a deformation of the generator $f$. 

\begin{Def}
Let $I=(f_1,\dots, f_p)$ be an ideal with a given minimal generator set. If $I'$ is a first order deformation of $I$, and if $g\in I'$ is an element on the form $g=f_i+th_i$, we say that $g$ is a first order deformation of $f_i$.
\end{Def}

Using this termonology, we give two lemmas that will help us find all deformations of $P_d$.

\begin{Lemma}\label{pushdeformation}
If $f$ and $f' = f\cdot \frac{x_{ik}}{x_{jl}}$ are both generators of $P_d$, and if $f+tm$ is a first order deformation of $f$, for some monomial $m$ of degree $d$. Then either $x_{ik}\cdot m\in P_d$, i.e. $x_{ik}\cdot m$ is divisible by a generator of $P_d$, and we say that $m$ vanishes in $f'$, or otherwise $x_{jl}$ divides $m$ and $g'=f'+tm'$ must be a deformation of $f'$, where $m' = m\cdot \frac{x_{ik}}{x_{jl}}$. In this case we say that $m$ is pushed to $f'$.
\end{Lemma}

\begin{proof}
If $f+tm$ is a generator of an ideal that is a first order deformation of $P_d$, we know from Lemma \ref{liftrelations} that all relations of $P_d$ lift to relations of this new ideal. There is a relation $x_{ik}\cdot f - x_{jl} \cdot f'$ in $P_d$. In order to get a relation of the deformation $g=f+tm$, we need either that $x_{ik}\cdot m$ is divisible by a generator of $P_d$ or that $g' = f'+tm'$.
\end{proof}

\begin{Lemma}
Let $f$ and $f'$ be as in Lemma \ref{pushdeformation} above. If there is a first order deformation $J$ of $P_d$, such that $f + tm_1+tm_2\in J$ is a first order deformation of $f$, then there is a first order deformation $J'$ of $P_d$ such that $f+tm_1\in J'$ is a first order deformation of $f$.
\end{Lemma}

\begin{proof}
Just as in the proof of Lemma \ref{pushdeformation} above, we look at the relation $x_{ik}\cdot f - x_{jl} \cdot f'$ in $P_d$. In order to get a relation of the deformation $g=f+tm_1+tm_2$, we must have one of the following:
\begin{itemize}
\item[i)]
$x_{ik}\cdot(m_1+m_2)\in P_d$. 
\item[ii)]
$x_{ik}\cdot m_1 \in P_d$ and $g'=f'+tm_2$.
\item[iii)]
$x_{ik}\cdot m_2 \in P_d$ and $g'=f'+tm_1$.
\item[iv)]
$g'=f'+tm_1+tm_2$.
\end{itemize} 
If ii), iii) or iv) is the case it follows immediately that we can just set $m_2 = 0$ and the first order deformation $J'$ can be built. But this is also possible if i) is the case, because if $x_{ik}\cdot(m_1+m_2)\in P_d$, we must have $x_{ik}\cdot m_1 \in P_d$ since $P_d$ is a monomial ideal.

\end{proof}

We want to study the first order deformations which are generated by deformations $g=f+tm$ for a monomial of degree $d$.

\begin{Lemma}\label{PushToVertex}
Let $P_d$ be a square-free polarization of $m^d$, which means that the vertices of $P_d$ are the monomials $v_i=x_{i1}x_{i2}\cdots x_{id}$ for $i=1,\dots,n$. If $g=f+tm$ is a first order deformation of a generator of $P_d$, then it can be pushed successively to at most one vertex of $P_d$.
\end{Lemma}

\begin{proof}
Suppose that $f = x_{i_1j_1}\cdots x_{i_dj_d}$ is a generator of $P_d$, and suppose that there is a deformation $f+tm$ for a monomial $m$. Suppose now that we can push the deformation (as described in Lemma \ref{pushdeformation}) successively into two vertices $v_a=x_{a1}x_{a2}\cdots x_{ad}$ and $v_b=x_{b1}\cdots x_{bd}$ for $a\neq b$. This means that $m$ must be divisible by all monomials $x_{i_tj_t}$ that divides $f$ such that $i_t\neq a$. On the other hand, $m$ must be divisible by all monomials $x_{i_tj_t}$ that divides $f$ such that $i_t\neq b$. But that means that $m$ must be divisible by all of these variables, so $m$ is divisible by $f$, and $f+tm$ is not a proper first order deformation. So the result follows. 
\end{proof}

We now introduce the two natural polarizations of $m^d$.

\begin{Def}
Let $m^d$ be the $d$-th power of the maximal ideal in the polynomial ring $S=k[x_1,\dots,x_n]$. The polarization
$$B_{nd} = \left( x_{i_11}x_{i_22}\cdots x_{i_dd} \, \middle| \, 1\le i_1 \le i_2 \le \cdots \le i_d \le n\right)$$
in $\widetilde{S} = k[x_{11},\dots,x_{1d},\dots,x_{n1},\dots, x_{nd}]$ is called the box polarization of $m^d$. 
\end{Def}

\begin{Def}
The polarization
$$P_{nd} = \left( x_{11}\cdots x_{1i_1}\cdot x_{21}\cdots x_{2i_2}\cdots x_{n1}\cdots x_{ni_n}\middle|i_j\ge 0 \text{ and } i_1+i_2+\cdots i_n = d\right)$$
is called the standard polarization of $m^d$. 
\end{Def}
The fact that the standard polarization is a polarization is well known and straight forward to check. The fact that $B_{nd}$ is a polarization can be found in \cite[Theorem 3.12]{NR2009Betti} or \cite[Theorem 3.4]{yanagawa2012alternative}.
\medskip

\begin{Example}
The ideal $P_{32}$ is the ideal
$$P_{32}=\left(x_{11}x_{12}, x_{11}x_{21}, x_{11}x_{31}, x_{21}x_{22}, x_{21}x_{31}, x_{31}x_{32}\right)$$
\end{Example}
\medskip

\begin{Example}
The ideal $B_{33}$ is the ideal
\begin{align*}
B_{33} = & \left(x_{11}x_{12}x_{13}, x_{11}x_{12}x_{23}, x_{11}x_{12}x_{33}, x_{11}x_{22}x_{23}, x_{11}x_{22}x_{33},\right. \\ 
& \left. x_{11}x_{32}x_{33}, x_{21}x_{22}x_{23}, x_{21}x_{22}x_{33}, x_{21}x_{32}x_{33}, x_{31}x_{32}x_{33}\right)
\end{align*}
\end{Example}

\begin{Lemma}\label{pushtoatleastonevertex}
Every first order deformation of a generator of $B_{nd}$ and $P_{nd}$ can be pushed to exactly one vertex.
\end{Lemma}

The proof of this Lemma is not hard, but quite technical. We therefore illustrate the idea of the proof with an example before giving the proof.
\medskip

\begin{Example}
Suppose that the generator $f = x_{11}x_{22}x_{33}$ of $B_{33}$ is deformed to $x_{11}x_{22}x_{33}+tm$ in a first order deformation of $B_{33}$. We know that there are relations
$$
\begin{array}{l}
x_{22}x_{33}\cdot x_{11}x_{12}x_{13} - x_{12}x_{13}\cdot f,  \\
x_{11}x_{33}\cdot x_{21}x_{22}x_{23} - x_{21}x_{23}\cdot f, \text{ and} \\
x_{11}x_{22}\cdot x_{31}x_{32}x_{33} - x_{31}x_{32}\cdot f
\end{array}
$$
in $B_{33}$. If $f+tm$ is a deformation of $f$, then we know, as in Lemma \ref{pushdeformation}, that either the deformation vanishes in all relations or it can be pushed into a vertex. If we assume that it vanish for all the relations above, it means that $x_{12}x_{13}\cdot m$, $x_{21}x_{23}\cdot m$ and $x_{31}x_{32}\cdot m$ are all divisible by generators of $B_{33}$. Since we assume that $f+tm$ is a proper deformation, we assume that $m$ is not divisible by a generator of $B_{33}$. We now observe that since $x_{12}x_{13}\cdot m$ is divisible by a generator, while $m$ is not, it means that the generator that divides the product must be divisible by $x_{12}$ or $x_{13}$ (or both). But by the construction of $B_{33}$, we see that the only generators which satisfies this is also divisible by $x_{11}$. A similar argument shows that $m$ also has to be divisible by $x_{33}$. An finally, we observe that since $x_{21}x_{23}\cdot m$ is divisible by a generator, it means that $m$ is divisible by $x_{j2}$ for at least one $j$ in the range $1\le j \le 3$. This means that $m$ is divisible by $x_{11}x_{j2}x_{33}$, but this is already a generator of $B_{33}$, so we reach a contradiction.
\end{Example}

\begin{proof}[Proof of Lemma \ref{pushtoatleastonevertex}]
In view of Lemma \ref{PushToVertex}, it is enough to show that if $m$ is a proper deformation of a generator $f$, then $m$ can be pushed to at least one vertex. First we consider the box polarization $B_{nd}$.

Let $f = x_{i_11}x_{i_22}\cdots x_{i_dd}$ be a generator of $B_{nd}$, and suppose that there is a deformation $f+tm$ for a monomial $m$. Suppose that $m$ vanishes when pushed toward all vertices $x_{a1}\cdots x_{an}$. This means that 
$$\left(\prod_{\{j\,|\,i_j \neq a\}} x_{a j}\right) \cdot m$$
is divisible by a generator for all $a$ such that $1\le a\le n$. But as we will see, this can only happen if $m$ is divisible by a generator.

First of all, let $l_1$ be the greatest index such that $i_{l_1}=i_1$. So $i_{l_1}<i_j$ for $l_1<j$. Then, we have that $\left(\prod_{\{j\,|\,i_j \neq i_{l_1}\}} x_{i_{l_1} j}\right) \cdot m = x_{i_{l_1} (l_1+1)}\cdots x_{i_{l_1} d} \cdot m$ is divisible by a generator. But this means that $m$ is divisible by a monomial $x_{b_1 1}\cdots x_{b_{l_1} l_1}$ such that $b_1\le b_2 \le \cdots \le b_{l_1} \le i_{l_1}$. Next, let $l_2$ be the greatest index such that $i_{l_2}=i_{l_1+1}$. Here we also have that $\left(\prod_{\{j\,|\, i_j \neq i_{l_2}\}} x_{i_{l_2} j}\right) \cdot m = x_{i_{l_2} 1}\cdots x_{i_{l_2} l_1}\cdot x_{i_{l_2} (l_2+1)}\cdots x_{i_{l_2} d} \cdot m$ is divisible by a generator. This can only happen if one of the following is the case:

\begin{itemize}
\item[1.]
$m$ is divisible by $x_{b_{l_1+1} (l_1+1)}\cdots x_{b_d d}$ where $i_{l_2}\le b_{l_1+1}\le \cdots \le b_d$. But if we combine this with the first case, we get that $m$ is divisible by $x_{b_1 1}\cdots x_{b_d d}$, where $b_1\le \cdots \le b_{l_1} \le i_{l_1} \le i_{l_2} \le b_{l_1+1}\le \cdots \le b_d$. In this case $x_{b_1 1}\cdots x_{b_d d}$ is a generator of $B_{nd}$, and $m$ is not a proper deformation.

\item[2.]
$m$ is divisible by $x_{c_1 1}\cdots x_{c_{l_2} l_2}$ where $c_1\le c_2 \le \cdots \le c_{l_2} \le i_{l_2}$.

\item[3.]
$m$ is divisible by $x_{c_{l_1+1} l_1+1}\cdots x_{c_{l_2} l_2}$ where $i_{l_1}\le c_{l_1+1} \cdots \le c_{l_2} \le i_{l_2}$. Combining this result with the case above, we get that $m$ is divisible by $x_{b_1 1}\cdots x_{b_{l_1} l_1}\cdot x_{c_{l_1+1} l_1+1}\cdots x_{c_{l_2} l_2}$, where $b_1\le \cdots b_{l_1} \le c_{l_1+1}\le c_{l_2}\le i_{l_2}$.
\end{itemize}

So, if 1. is the case, then $m$ is divisible by a generator of $B_{nd}$. If either 2. or 3. is the case, then we get that $m$ is divisible by  $x_{c_1 1}\cdots x_{c_{l_2} l_2}$ where $c_1\le c_2 \le \cdots \le c_{l_2} \le i_{l_2}$.

Continuing the same way, we may define $l_t$ for all $t$ untill $l_t = i_d$. But then we get that either $m$ is divisible by a generator of $B_{nd}$ as in 1. above, or that $m$ is divisible by $x_{c_1 1}\cdots x_{c_{i_d} d}$ for a sequence $c_1\le \cdots c_d \le i_d$. But this is a generator of $B_{nd}$

Next, we show that the same result holds for the standard polarization $P_{nd}$.

Let $f = x_{i_11}\cdots x_{i_1d_1}x_{i_21}\cdots x_{i_2d_2}\cdots x_{i_r1}\cdots x_{i_rd_r}$ be a generator of $P_{nd}$, and suppose that there is a deformation $f+tm$ for a monomial $m$. Suppose that $m$ vanishes when pushed toward all vertices $x_{a1}\cdots x_{an}$. But this means that $\left(\prod_{\{j\,|\,x_{aj}\nmid m\}} x_{a j}\right) \cdot m$ is divisible by a generator of $P_{nd}$, for all $a$ such that $1\le a\le n$. In particular, that means that $x_{i_j d_j+1}\cdots x_{i_j d}m$ is divisible by a generator for all $1\le j \le r$. But this is only possible if either $m$ is divisible by a generator, or if $m$ is divisible by $x_{i_j 1}\cdots x_{i_j d_j}$ for all $1\le j\le r$. Since we assume that $m$ is a proper deformation of $f$, we must have that $m$ is divisible by $x_{i_j 1}\cdots x_{i_j d_j}$ for all $1\le j\le r$. But again, that means that $m$ is divisible by $f$, which is a contradiction.
\end{proof} 

\section{The box polarization}\label{theboxpolarization}
We will now show that the box polarization corresponds to a smooth point on the Hilbert scheme of $M_d$. This will be done, by first showing that $B_{nd}$ lies on a component where the dimension is known. Then we calculate the dimension of the tangent space of $B_{nd}$ and show that this dimension is the same as the component.

\begin{Prop}\label{boxpolinit}
Let $M_{nd}$ be the matrix
$$
\left(
\begin{array}{ccccccc} 
x_{11} & x_{21} & \cdots & x_{n1} & 0 & \cdots & 0 \\
0 & x_{12} & x_{22} & \cdots & x_{n2} &  \cdots & 0 \\
\vdots &  & \ddots & & &\ddots & \vdots \\
0 & \cdots & 0  & x_{1d} & x_{2d} & \cdots & x_{nd}
\end{array}
\right),
$$
and let $I$ be the ideal generated by its $d\times d$-minors. Then
$$B_{nd} = \mathrm{in}_{<}(I)$$
for lexicographic term order $<$, with $x_{ij}<x_{i'j'}$ if $j<j'$, or $i<i'$ and $j=j'$. This means that $B_{nd}$ and $I$ lies on the same component of the Hilbert scheme of $M_d$.
\end{Prop} 

\begin{proof}
This follows from Theorem 1 and Lemma 5 in \cite{sturmfels1990grobner}. They show that the $d\times d$-minors of the matrix 
$$
\left(
\begin{array}{ccccccc} 
x_{11} & x_{21} & \cdots & x_{n1} & x_{(n+1) 1} & \cdots & x_{(n+d-1) 1} \\
x_{02} & x_{12} & x_{22} & \cdots & x_{n2} &  \cdots & x_{(n+d-2) 2} \\
\vdots &  & \ddots & & &\ddots & \vdots \\
x_{(2-d) d}  & \cdots & x_{0 d}  & x_{1d} & x_{2d} & \cdots & x_{nd}
\end{array}
\right),
$$
under the same term order, are a reduced Gr\"{o}bner basis and that the leading term of each $d\times d$-minors are exactly the generators of $B_{nd}$. So the $d\times d$-minors of $M_{nd}$ will also be a reduced Gr\"{o}bner basis, and the leading terms are also the generators of $B_{nd}$.

\end{proof}

\begin{Prop}\label{dimcomponent}
Let $M$ be a $d\times (n+d-1)$-matrix with general linear entries from $\widetilde{S}$, and let $I$ be the ideal generated by its $d\times d$-minors. Then $I$ is a smooth point on the Hilbert scheme, and the dimension of its component is
$$d(d+n-1)nd-d^2-(d+n-1)^2+1.$$
\end{Prop}

\begin{proof}
This follows by Theorem 5.8 and Corollary 5.9 in \cite{kleppe2010deformations}. This explicit formula is also found in the corollary to the main theorem in \cite{faenzi2010hilbert}.
\end{proof}

We will show that $B_{nd}$ is a smooth point in the Hilbert scheme by showing that the tangent space has dimension equal the dimension of its component. By Lemma \ref{pushtoatleastonevertex} above, it is enough to count the dimensions of first order deformations in all vertices of $B_{nd}$. In order to count these first order deformations, we need a lemma:

\begin{Lemma}\label{vertexlemma}
Let $f_i = x_{i1}\cdots x_{id}$ denote the vertices in $B_{nd}$. If $f_1+tm$ is a deformation of $f_1$, then either $x_{2d}\cdot m$ is divisible by a generator of $B_{nd}$, or there is a $j$ such that $x_{1(d-j)}\cdots x_{1d}$ divides $m$ and $\frac{x_{2(d-j-1)}\cdots x_{2d}}{x_{1(d-j)}\cdots x_{1d}}\cdot m$ is divisible by a generator of $B_{nd}$. Similarly, if $f_n+tm$ is a deformation of $f_n$, then either $x_{(n-1)1}\cdot m$ is divisible by a generator of $B_{nd}$, or there is a $j$ such that $x_{n1}\cdots x_{nj}$ divides $m$ and $\frac{x_{(n-1)1}\cdots x_{(n-1)(j+1)}}{x_{n1}\cdots x_{nj}}\cdot m$ is divisible by a generator of $B_{nd}$.

If $1 < i < n$, and if $f_i+tm$ is a deformation of $f_i$, then $m$ is either divisible by $x_{i1}\cdots \widehat{x_{ij}}\cdots x_{id}$, or $m$ is
$$f_i\cdot \frac{x_{(i+1)j}}{x_{ij}}\cdot \frac{x_{(i-1)(j+1)}}{x_{i(j+1)}},$$
for one $1\le j < d$.

\end{Lemma}

\begin{proof}
The first claim follows by trying to push the deformation from $f_1$ to $f_2$. Then at some point the deformation will vanish. By the construction of $B_{nd}$, $m$ will also vanish if we try to push $m$ to any other vertex generator $f_i$. The second claim is similar.

Finally, if $1<i<n$, then we must have that both the two following: \\
Either $x_{(i+1)n}\cdot m$ is divisible by a generator of $B_{nd}$, or there is a $j$ such that $x_{i(n-j)}\cdots x_{in}$ divides $m$ and $\frac{x_{(i+1)(n-j-1)}\cdots x_{(i+1)n}}{x_{i(n-j)}\cdots x_{in}}\cdot m$ is divisible by a generator of $B_{nd}$. And: \\
Either $x_{(i-1)1}\cdot m$ is divisible by a generator of $B_{nd}$, or there is a $j$ such that $x_{i1}\cdots x_{ij}$ divides $m$ and $\frac{x_{(i-1)1}\cdots x_{(i-1)(j+1)}}{x_{i1}\cdots x_{ij}}\cdot m$ is divisible by a generator of $B_{nd}$. 

If both of these holds, then $m$ is either divisible by $x_{i1}\cdots\widehat{x_{ij}}\cdots x_{in}$ for some $j$. Or on the form written in the lemma.
\end{proof}

\begin{Theorem}\label{boxpolsmooth}
The ideal $B_{nd}$ is smooth in the Hilbert scheme.
\end{Theorem}

\begin{proof}
By Proposition \ref{boxpolinit} and \ref{dimcomponent} above, it is enough to show that the dimension of the tangent space of $B_{nd}$ in the Hilbert scheme is at most $d(d+n-1)nd-d^2-(d+n-1)^2+1$. By Lemma \ref{pushtoatleastonevertex} above, it is enough to calculate the dimension of the deformations of the vertex generators $f_i = x_{i1}\cdots x_{in}$. We start by calculating the dimension of deformations of the vertex generator $f_1$. For each of the cases in Lemma \ref{vertexlemma} above, we count how many different monomials $m$ are possible. 

If $x_{2d}\cdot m$ is divisible by a generator of $B_{nd}$, and $m$ is not, then $m$ has to be divisible by a monomial generator of the box polarization of $(x_1,x_2)^{d-1}$. Suppose that $m$ is not divisible by $x_{1d}$. We then count that there are
$$d(d-1)n-(d-1)$$
possible $m$ of degree $d$ satifying this which are not generators of $B_{nd}$. This is because there are $d$ monomial generators of the box polarization of $(x_1,x_2)^{d-1}$, and each of them can be multiplied by a variable $x_{ab}$ where $b<d$ to produce a monomial of degree $d$ which is not a generator of $B_{nd}$. This can therefore be done in $d(d-1)n$ ways. However, for each linear relation between the monomial generators of the box polarization of $(x_1,x_2)^{d-1}$, we see that the corresponding monomial of degree $d$ can be made in two different ways. Since there are $d-1$ such relations, we need to subtract $d-1$ as we have counted these monomial twice.

For each $j$ in the range $1<j\le d$, we now assume that $x_{1j}\cdots x_{1d}$ divides $m$, and that $\frac{x_{2(j-1)}\cdots x_{2d}}{x_{1j}\cdots x_{1d}} \cdot m$ is divisible by a generator of $B_{nd}$. So we have that $m$ has to be divisible by a monomial generator of the box polarization of $(x_1,x_2)^{j-2}$. If we suppose that $x_{1(j-1)}$ does not divide the monomial $m$, we calculate that there are
$$(j-1)(d-1)n-(j-2)+(j-1)(n-1)$$
monomials satisfying this property without being a generator of $B_{nd}$. The first part is done exactly as in the case above, and the extra part $(j-1)(n-1)$ comes from the $(n-1)$ variables $x_{t(j-1)}$ for $t=2,3,\dots,n$ which multiplied by the degree $d-1$-monomial we must have, in all $j-1$ cases, also gives possible a $m$ of degree $d$ not divisible by a generator. So, by summing up we get that the dimension of first order deformations of $f_1$ is at most
$$
\begin{array}{cl}
&\sum\limits_{i=1}^{d}i(d-1)n - \sum \limits_{i=1}^{d-1} i + \sum \limits_{i=1}^{d-1} i(n-1) \\
& \\
= & \frac{d(d-1)}{2}\left(nd+2n-2\right)
\end{array}
$$

By symmetry, we can also compute that the number of first order deformations of $f_n$ is the same in precisely the same way.

Finally, we want to compute the number of possible first order deformations of $f_i$ when $1<i<n$. So suppose that $m$ is divisible by $x_{i2}\cdots x_{id}$. Then there are
$$(d-1)n + (n-i)$$
monomials satisfying this. Here the $(d-1)n$ part comes from variables $x_{ab}$ with $b\neq 1$, and the $(n-1)$ part comes from variables $x_{a1}$ with $a>i$.
If $m$ is divisible by $x_{i1}\cdots x_{i(d-1)}$, then there are
$$(d-1)n+(i-1)$$
monomials satisfying this. And if $m$ is divisible by $x_{i1}\cdots\widehat{x_{ij}}\cdots x_{id}$, then there are
$$(d-1)n+(n-1)$$
monomials satisfying this. Otherwise, there is exactly $(d-1)$ monomials which are deformations of $f_i$.

So by summing up, we get that the dimension of the first order deformations of $f_i$ is at most
$$
\begin{array}{cl}
&\sum\limits_{i=1}^{d}(d-1)n - \sum \limits_{i=1}^{d-2} (n-1) + (n-i)+(i-1) + (d-1)\\
& \\
= & n(d+1)(d-1)
\end{array}
$$

So the dimension of the first order deformation of $B_{nd}$ is therefore at most
$$
\begin{array}{cl}
& 2\cdot \frac{d(d-1)}{2}\left(nd+2n-2\right) + (n-2)n(d+1)(d-1) \\
& \\
= & d(d+n-1)nd-d^2-(d+n-1)^2+1
\end{array}  
$$

This means that the dimension of the tangent space of $B_{nd}$ in the Hilbert scheme is at most the same as the dimension of the component it lies in. Hence, it is a smooth point.

\end{proof}

\section{Standard Polarization}\label{standardpolarization}

We will now show that the standard polarization $P_{nd}$ is a smooth point on the Hilbert scheme. This will be done by finding all first order deformations of $P_{nd}$. Then we show that all first order deformations can be lifted to global deformations.

Consider a first order deformation of the standard polarization $P_{nd}$. We want to descibe how the first order deformations of the vertices in $P_{nd}$ are built up.

\begin{Lemma}\label{stdpollemma}
Let $f_1 = x_{11}\cdots x_{1d}$, and let $f_1+tm$ be a first order deformation of $f_1$. Assume that $x_{1(j+1)}\cdots x_{1d}$ divides $m$ and $x_{1j}$ does not. If $n\ge 4$, or if $n=3$ and $j<d$ then $m$ is divisible by $x_{11}\cdots \widehat{x_{1j}}\cdots x_{1d}$. If $n=3$ and $j=d$ then either $m$ is divisible by $x_{11}\cdots x_{1(d-1)}$ or $m$ is divisible by $x_{11}\cdots x_{1(d-2)}x_{22}x_{32}$.
\end{Lemma}

\begin{proof}
First, we will show the result when $n\ge 4$ and $j=d$. So assume that $x_{1d}$ does not divide $m$. Then, by Lemma \ref{pushdeformation}, we have that $x_{j1}m$ are all divisible by a generator of $P_{nd}$ for $2\le j\le n$. But since $m$ is not divisible by a generator, this means that none of the variables $x_{j1}$ divides $m$. This can only happen if $m$ is divisible by $x_{11}\cdots x_{1(d-1)}$. The only other possibility would be that $m$ is divisible by $x_{11}\cdots x_{1d'}x_{22}\cdots x_{2d_2} \cdots x_{n2}\cdots x_{nd_n}$, where $d_j\ge 2$ and $d'+d_j\ge d$ for all $j=2,\dots,n$. On the other hand, we must also have that $d'+(d_2-1)+(d_3-1)+\cdots+(d_n-1) \le d.$ So this means in particular that
$$d'+(d_2-1)+(d_3-1)+\cdots+(d_n-1) \le d\le d'+d_2,$$ 
which means that 
$$-1+(d_3-1)+(d_4-1)+\cdots+(d_n-1)\le 0.$$
But if $n\ge 4$, the left hand side is positive since $d_3\ge 2$ and $d_4\ge 2$, so we may exclude this case. 

Next, we show the case where $n=3$ and $j=d$. Suppose that $x_{1d}$ does not divide $m$. That means that $mx_{21}$ and $mx_{31}$ are both divisible by generators of $P_{3d}$, while $m$ is not. Exactly similar as above, we see that two cases given in the lemma are the only possibilities.

Finally, we assume that $n\ge 3$ and $j<d$. So we have that $x_{1j}$ does not divide $m$ while $x_{1(j+1)}\cdots x_{1d}$ does. By pushing the deformation towards the vertex $f_2$, we get that $mx_{21}\cdots x_{2(d-j+1)}$ is divisible by a generator of $P_{nd}$ and by pushing the deformation towards the vertex $f_3$, we get that $mx_{31}\cdots x_{3(d-j+1)}$ is divisible by a generator of $P_{nd}$. Again, we get two possible cases. Either $m$ is divisible by $x_{11}\cdots x_{1(j-1)}$, and we are done. Or, we would have that $m$ is divisible by $x_{11}\cdots x_{1(j-2)}x_{2(d-j+2)}x_{3(d-j+2)}$. However, we can see that this last case is not possible by trying to push the deformation first towards the vertex $f_2$, then towards the vertex $f_3$. By doing this, we must also have that $mx_{21}\cdots x_{2(d-j)}x_{31}$ is divisible by a generator of $P_{nd}$. And by pushing the deformation first towards the vertex $f_3$ then towards the vertex $f_2$, we must have that $mx_{31}\cdots x_{3(d-j)}x_{21}$ is divisible by a generator of $P_{nd}$. These two new conditions means that $m$ also must be divisible by $x_{22}x_{32}$. However, adding all the conditions together would mean that $m$ has degree at least $d+1$, but this is a contradiction, since we assumed $m$ to be of degree $d$.

\end{proof}

We can now show that the standard polarization is a smooth point in the Hilbert scheme of $M_d$. This is because for most cases, the only deformations of $P_{nd}$ are deformations of the variables. For the special cases where we have some other first order deformation, we will show explicitely that they lift to global deformations.

\begin{Theorem}
The ideal $P_{nd}$ is smooth in the Hilbert scheme of $M_d$.
\end{Theorem}

\begin{proof}
We wish to find all first order deformations of the ideal $P_{nd}$. By Lemma \ref{pushtoatleastonevertex}, it is enough to find the first order deformation in the vertices of $P_{nd}$. By the symmetry of the standard polarization, the deformations of the vertex $f_i$ can be found similar as for $f_1$.

If $n\ge 4$, then by Lemma \ref{stdpollemma} above, we see that all deformations $f_i+tm$ have $m$ divisible by $x_{i1}\cdots \widehat{x_{ij}}\cdots x_{id}$. But such a deformation is just a deformation of the variable $x_{ij} \mapsto x_{ij}+tx_{i'j'}$. Since these are all first order deformations of $P_{nd}$, and since it is well known that these deformations lift to global deformations, we get that $P_{nd}$ is smooth in the Hilbert scheme.

If $n=3$, then by Lemma \ref{stdpollemma} above, we get that the deformations $f_i+t_im_i$ are either just deformations of the variables as above, or we must have that $m_1 = x_{11}\cdots x_{1(d-2)}x_{22}x_{32}$, $m_2 = x_{21}\cdots x_{2(d-2)}x_{12}x_{32}$ and $m_3=x_{31}\cdots x_{3(d-2)}x_{12}x_{22}$.

Since the deformations of the variables lift to global deformations, it is enough to verify that these three first order deformations also lift to global deformations.

We show this by lifting the deformation $x_{11}\cdots x_{1d}+tx_{11}\cdots x_{1(d-2)}x_{22}x_{32}$. The relations in the ideal $P_{nd}$  are $x_{11}\cdots x_{1d}\cdot x_{21} - x_{11}\cdots x_{1(d-1)}x_{21}\cdot x_{1d}$ and $x_{11}\cdots x_{1d}\cdot x_{31} - x_{11}\cdots x_{1(d-1)}x_{31}\cdot x_{1d}$. If $d\ge 3$ we see that these relations lifts to the relations
$$
\begin{array}{rl}
&(x_{11}\cdots x_{1d}+tx_{11}\cdots x_{1(d-2)}x_{22}x_{32})\cdot x_{21} \\ 
- & (x_{11}\cdots x_{1(d-1)}x_{21})\cdot x_{1d} \\
- & (x_{11}\cdots x_{1(d-2)}x_{21}x_{22})\cdot tx_{32}
\end{array}
$$ and
$$
\begin{array}{rl}
&(x_{11}\cdots x_{1d}+tx_{11}\cdots x_{1(d-2)}x_{22}x_{32})\cdot x_{31} \\
- & (x_{11}\cdots x_{1(d-1)}x_{31})\cdot x_{1d} \\
- &(x_{11}\cdots x_{1(d-2)}x_{31}x_{32})\cdot tx_{22}.
\end{array}
$$
These relations contains only one vertex $x_{11}\cdots x_{1d}$. We can therefore lift all the relations in the ideal $I$, where $x_{11}\cdots x_{1d}+t_1x_{11}\cdots x_{1(d-2)}x_{22}x_{32}$, $x_{21}\cdots x_{2d}+t_2x_{21}\cdots x_{2(d-2)}x_{12}x_{32}$ and $x_{31}\cdots x_{3d}+t_3x_{31}\cdots x_{3(d-2)}x_{12}x_{22}$ are the deformations of the vertices, simultaneously.
 
The special case where $n=3$ and $d=2$ can also be liftet to global deformation. It is straight forward to verify that the relations of the ideal $P_{3d}$ can be lifted to relations of the ideal

\begin{align*}
(x_1x_2+t_1y_2z_2,\, & x_1y_1-t_1t_2z_2^2,\, x_1z_1-t_1t_3y_2^2, \\
& y_1y_2+t_2x_2z_2,\, y_1z_1-t_2t_3x_2^2,\, z_1z_2+t_3x_2y_2).
\end{align*}

Finally, the case where $n=2$. But here it is easy to see that $P_{2d}$ is isomorphic to the box polarization $B_{2d}$ via the isomorphism sending $x_{1i}$ to $x_{1i}$ and $x_{2j}$ to $x_{2,d-j+1}$. Since we already know that $B_{2d}$ is smooth, it follows that $P_{2d}$ is smooth.

\end{proof}

\section{The ideal $m_\mathrm{sq.fr.}^d$}\label{kvadratfriversjon}
The main purpose of polarizing monomial ideals is to find a square-free monomial ideal with the same properties, i.e. same Betti numbers and Betti table, etc., and use theory from combinatorial commutative algebra to find the resolution, Betti numbers, etc. However, the polarized ideal has many variables which sometimes makes this hard. It may therefore be useful to find a square-free monomial ideal with as few variables as possible, which has the same properties, i.e. Betti numbers, etc. For the ideal $m^d$, we can produce a square-free monomial ideal which we denote by $m_\mathrm{sq.fr.}^d$ with this property.

\begin{Def}   
Let $k[x_1,\dots,x_n]$ be the polynomial ring in $n$ variables, and let $m^d$ be the $d$-th power of the maximal ideal. Then we define
$$m_\mathrm{sq.fr.}^d := \left(x_{i_1}\cdots x_{i_d}\,|\, 1\le i_1<i_2<\cdots <i_d\le n+d-1\right),$$
to be the $d$-th square-free power of the maximal ideal in the polynomial ring $k[x_1,\dots,x_{n+d-1}]$.
\end{Def}

\begin{Remark}
The fact that these monomial ideals have the same Betti numbers is because they can both be polarized by the box polarization. This is used in e.g. \cite{NR2009Betti} to produce a minimal free resolution of the ideal $m^d$. 
\end{Remark}
\medskip

It is also interesting to study different polarizations of the ideal $I_d=m_\mathrm{sq.fr.}^d$. This ideal have generally many more polarizations than the ideal $m^d$, and the combinatorics describing these polarizations are studied in \cite{lohne2013polarizations}. In this setting, and when $d\ge 3$, we don't have such a nice choice of polynomial ring where all the polarizations lie. In fact, if we study the so-called maximal polarizations, i.e. polarizations which have no non-trivial polarizations themself, it is not clear which polynomial rings we should use to include all polarization. It is therefore quite difficult to study the deformations of different polarizations in the same way as we did for $m^d$. However, when $d=2$, then all maximal polarizations lie in the same polynomial ring. In this case, we also have that every maximal polarization of $m^2$ is also a maximal polarization of $m_{\mathrm{sq.fr}^2}$ (see \cite[Proposition 4.7]{lohne2013polarizations}). We will study the deformations of the ideals in this case. The combinatorics describing the maximal polarizations are especially nice when $d=2$. We will use this description to calculate the dimension of the tangent space of each polarization in this case.

Let $n'\ge 3$ be an integer, and consider the ideal $m_\textrm{sq.fr.}^d$ consisting of all square-free monomials of degree $2$ in $k[x_1,\dots, x_{n'}]$. Then the following result is obtained in \cite{lohne2013polarizations}.

\begin{Theorem}\label{treetheorem}
Every spanning tree $T$ of $K_{n'}$, i.e. the complete graph on $n'$ vertices, corresponds to a maximal polarization of $m_\textrm{sq.fr.}^2$ in the following way. We name the edges in $T$ by $e_t$ for $t=1,\dots,n'-1$. Then $I_T$ is obtained as follows
$$I_T = \left(x_{ia}x_{jb}\,|\, i<j\text{ and the path from } i \text{ to } j \text{ starts in } e_a \text{ and ends in }e_b\right)$$
Furthermore, every maximal polarization of $m_\textrm{sq.fr.}^2$ comes from such a tree.
\end{Theorem}

\begin{Example}\label{trepolarisering}
Let $T$ be the spanning tree
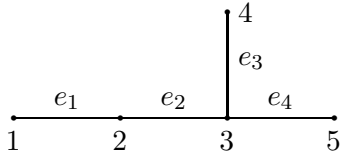
\begin{center}
\begin{picture}(150,60)
\put(10,12){\circle*{2}}
\put(7,0){$1$}
\put(25,17){$e_1$}
\put(10,12){\line(1,0){40}}
\put(50,12){\circle*{2}}
\put(47,0){$2$}
\put(65,17){$e_2$}
\put(50,12){\line(1,0){40}}
\put(90,12){\circle*{2}}
\put(87,0){$3$}
\put(90,12){\line(0,1){40}}
\put(90,52){\circle*{2}}
\put(94,48){$4$}
\put(94,32){$e_3$}
\put(130,12){\circle*{2}}
\put(105,17){$e_{4}$}
\put(90,12){\line(1,0){40}}
\put(127,0){$5$}
\end{picture}
\end{center}
of the complete graph on $5$ vertices. Then

\begin{align*}
I_T = (x_{11}x_{21}, x_{11}x_{32}, & x_{11}x_{43},x_{11}x_{54}, x_{22}x_{32}, \\
& x_{22}x_{43}, x_{22}x_{54}, x_{33}x_{43}, x_{34}x_{54}, x_{43}x_{54}). 
\end{align*}

\end{Example}

In the case of deformations of polarizations of the maximal ideal, we showed that it is enough to calculate the number of different deformations of the generators that corresponded to the vertices of $m^d$. In the case of $m_\textrm{sq.fr.}^2$ we no longer have these vertices and the situation seems to be more difficult. However, the problem is solved by the using the description of the polarization given in Theorem \ref{treetheorem} above.

\begin{Def}
Suppose that $I_T$ is a polarization of $m_\textrm{sq.fr.}^2$, corresponding to the tree $T$, and suppose that $T$ has the edges $e_t$, for $t=1,\dots, n'-1$. Then the monomials $f_t = x_{it}x_{jt}$, where $e_t = (i,j)$ are called the vertices of $I_T$.
\end{Def}  

\begin{Lemma}
Let $I_T$ be a polarization of $m_\textrm{sq.fr.}^2$ corresponding to the tree $T$. Then every first order deformation of a generator of $I_T$ can be pushed to exactly one vertex.
\end{Lemma}

\begin{proof}
Suppose that $f = x_{ia}x_{jb}$ is a generator of $I_T$, and suppose that $f+tm$ is a deformation of the generator. We first show that $m$ can not vanish when pushed towards every vertex of $I_T$. It is clear that it vanishes when pushed towards any vertex except $e_a$ and $e_b$. So it is enough to assume that it vanishes when pushed towards $e_a$ and $e_b$. Assume that $e_a=(i,i')$ and $e_b = (j',j)$. If we assume that $m$ vanishes when pushed towards the vertex $x_{ia}x_{i'a}$, we must have that $m\cdot x_{i'a}$ is divisible by a generator $f'$ of $I_T$. Since $m$ is assumed to not be divisible by a generator, we must have that $f'$ is divisible by $x_{i'a}$. Since the generators of $I_T$ correspond to paths, we see that $f'$ must correspond to a path starting in $i'$ moving away from $j$. So $m$ is divisible by a variable $x_{vk}$, where $e_k$ is the first edge in the path from $v$ to $j$ via $i$. Similarly, by pushing the deformation towards the other vertex, we get that $m$ is divisible by a variable $x_{wl}$, where $e_l$ is the last edge in the path from $i$ to $w$ via $j$. But that means that $m$ is divisible by $x_{vk}x_{wl}$. However, this is a generator of $I_T$ since the path from $v$ to $w$ starts in $e_k$ and ends in $e_l$. So we get a contradiction.

It remains to show that $m$ can not be pushed to both vertices corresponding to $e_a$ and $e_b$. But this is straight forward to verify, since $m$ must be divisible by $x_{jb}$ if it can be pushed to $x_{ia}x_{i'a}$ and it must be divisible by $x_{ia}$ if it can be pushed to $x_{j'b}x_{jb}$. Hence it must be divisible by $f$ itself, which is a contradiction.
\end{proof}

\begin{Def}
Let $T$ be a spanning tree of $K_{n'}$. We form the graph $G_T$ with vertex set consisting of edges of $T$, and the edges in $G_T$ corresponds to adjacent edges in $T$. Let $\nu_i^T = \#\{v\in G_T\,|\,\mathrm{deg}(v) = i\}$. We define the index of $T$ to be $i(T) = (n'-2)\nu_1^T + \nu_2^T$.
\end{Def}

\begin{Example}
If $T$ is the spanning tree in Example \ref{trepolarisering} above, then $G_T$ is the graph:
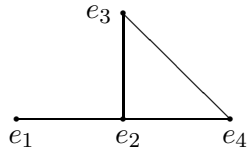
\begin{center}
\begin{picture}(100,60)
\put(10,12){\circle*{2}}
\put(7,3){$e_1$}
\put(10,12){\line(1,0){40}}
\put(50,12){\circle*{2}}
\put(47,3){$e_2$}
\put(50,12){\line(0,1){40}}
\put(50,52){\circle*{2}}
\put(50,52){\line(1,-1){40}}
\put(36,50){$e_3$}
\put(90,12){\circle*{2}}
\put(50,12){\line(1,0){40}}
\put(87,3){$e_4$}
\end{picture}
\end{center}
We can calculate that $i(T) = (5-2)\cdot 1 + 2 = 5$.
\end{Example}

\begin{Theorem}
Let $T$ be a spanning tree of $K_{n'}$ and suppose that $I_T$ is the corresponding polarization of $m_\textrm{sq.fr.}^2$. Then the dimension of the tangent space of $I_T$ in the Hilbert scheme of $m_\textrm{sq.fr.}^2$ is
$$(3n'-4)(n'-1)+i(T).$$
Furthermore, these are all smooth points in the same Hilbert scheme.
\end{Theorem}

\begin{proof}
We first show that if $g=x_{ik}x_{jk}$ is a generator of $I_T$ corresponding to the edge $e_k=(i,j)$, then the number of first order deformations of $g$ corresponding to deformations of the variables is $(3n'-4)$. The edge $e_k$ disconnects $T$ into two parts. Let $L$ be the component containing the vertex $i$ and let $R$ be the component containing the vertex $j$. We calculate that the dimension of deformations of the variable $x_{jk}$ is the total number of variables minus the number of vertices in $R$, and that the dimension of deformations of the variable $x_{ik}$ is the total number of variables minus the number of vertices in $L$. So the number of first order deformations of $g$, which corresponds to deformation of the variables is:
$$2(n'-1)-|R|+2(n'-1)-|L| = 4(n'-1)-n'=3n'-4.$$
Next, we need to calculate the dimension of deformations that does not occur as a deformation of the variables. We assume that $g = x_{ik}x_{jk}+mt$ is a first order deformation of a generator corresponding to the edge $e_k=(i,j)$. We observe that every other vertex of $T$ gives rise to a linear relation of $I_T$. This is of course because if $v$ is another vertex, then one of the paths $i$ to $v$ or $j$ to $v$ has to either start or end in $e_k$. Since $m$ is assumed to be a deformation not coming from deformation of the variables, it must vanish on these relations. We get the criteria that $m$ has to be divisible by a variable for each of the connected components in the graph $T\setminus\{i,j\}$. If $e_k$ has only one adjacent edge, i.e. $e_{l} = (j,j')$, it means that the only criterion for $m$ is that it is divisible by $x_{jl}$. The number of possible monomials of degree two, satisfying this, which are not a deformation of the variables, is $2(n'-1)-2 = n'-2$. If $e_k$ has two adjacent edges, i.e. $e_{l} = (j,j')$ and $e_{l'}=(i,i')$, then the criteria for $m$ is that it must be divisible by $x_{jl}$ and $x_{il'}$. There is only one monomial of degree $2$ satisfying this. Hence, the dimension of the first order deformations of $I_T$, which are not deformations of the variables are calculated by the function $i(T)$.

Since there are no linear relations between two vertices of $I_T$, and since every first order deformation of a vertex must vanish in all linear relations, it follows that all the first order deformations lift to global deformations of $I_T$. We therefore have that every such polarization is a smooth point in the Hilbert scheme.

\end{proof}

\begin{Remark}
In \cite{lohne2013polarizations}, we show that the box polarization corresponds to the line graph, while the standard polarization corresponds to the star graph. One might therefore assume that these two polarizations would give the highest and lowest dimensional tangent space in the Hilbert scheme. Indeed, the standard polarization will have the lowest dimensional tangent space, this is because $i(T)=0$ for $n'\ge 5$ since $G_T$ is the graph $K_{n'-1}$. When $n'<5$, then the only maximal polarizations are the box polarization and the standard polarization, and the calculations of the tangens spaces show that the standard polarization has the lowest dimension.

However, the box polarization is surprizingly not the polarization with the highest dimensional tangent space. For instance, if $n'=7$ and $T$ is the graph   

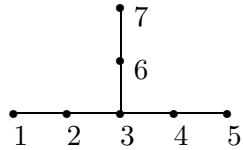
\begin{center}
\begin{picture}(80,60)
\put(40,12){\circle*{3}}
\put(0,0){$1$}
\put(20,0){$2$}
\put(40,0){$3$}
\put(60,0){$4$}
\put(80,0){$5$}
\put(45,25){$6$}
\put(45,45){$7$}
\put(0,12.25){\line(1,0){80}}
\put(40.25,12){\line(0,1){40}}
%\put(0,31){$_{e_1}$}
\put(0,12){\circle*{3}}
\put(20,12){\circle*{3}}
%\put(0,49){$1$}

%\put(10,12){\line(2,5){14}}
%\put(17,27){$_{e_2}$}
\put(60,12){\circle*{3}}
\put(80,12){\circle*{3}}
%\put(28,45){$2$}

%\put(10,12){\line(1,0){40}}
%\put(23,7){$_{e_{n-1}}$}

%\put(53,10){$n-1$}
\put(40,32){\circle*{3}}
\put(40,52){\circle*{3}}

%\put(25,22){$\ddots$}

\end{picture}
\end{center}
then $i(T) = 3\cdot 5 = 15$, while $i(L) = 2\cdot 5 + 4\cdot 1 = 14$ for the line graph $L$ corresponding to the box polarization.
\end{Remark}

\bibliography{hlohne}

\end{document}